\documentclass{amsart}
\usepackage{amsmath, amscd, amssymb, amsthm}
\usepackage{bbm}
\usepackage{latexsym}
\usepackage{amsfonts}
\usepackage{graphicx}
\usepackage[all,cmtip]{xy}
\usepackage[colorlinks,linkcolor=blue,breaklinks=true,urlcolor=blue,citecolor=blue,anchorcolor=blue,pagebackref]{hyperref}%
\usepackage{geometry}
\newtheorem{theorem}{Theorem}
\newtheorem{lemma}{Lemma}
\newtheorem{corollary}[theorem]{Corollary}

\newtheorem{proposition}{Proposition}

\newtheorem{conjecture}{Conjecture}

\newcommand{\C}{{\mathbb C}}

\renewcommand*\backref[1]{}
\renewcommand*\backrefalt[4]{ \ifcase #1 \or (cited on page #2) \else (cited on pages #2) \fi}

\newcommand{\be}{\begin{equation}}
\newcommand{\ee}{\end{equation}}
\newcommand{\bea}{\begin{eqnarray}}
\newcommand{\eea}{\end{eqnarray}}

\def\XXint#1#2#3{{\setbox0=\hbox{$#1{#2#3}{\int}$ }
\vcenter{\hbox{$#2#3$ }}\kern-.6\wd0}}

\numberwithin{equation}{section}
\allowdisplaybreaks

\begin{document}

\title[Locally conformally K\"ahler manifolds with constant Levi-Civita or Bismut holomorphic sectional curvature]
{Locally conformally K\"ahler manifolds with constant Levi-Civita or Bismut holomorphic sectional curvature}

\author{Shuwen Chen}
\address{Shuwen Chen. School of Mathematical Sciences, Chongqing Normal University, Chongqing 401331, China}
\email{{3153017458@qq.com}}\thanks{Zheng is the corresponding author. He is partially supported by National Natural Science Foundations of China
with the grant No. 12471039 and  12141101, and is supported by the 111 Project D21024.}

\author{Fangyang Zheng}
\address{Fangyang Zheng. School of Mathematical Sciences, Chongqing Normal University, Chongqing 401331, China}
\email{20190045@cqnu.edu.cn; franciszheng@yahoo.com} \thanks{}

\date{}

\begin{abstract}
An old conjecture in non-K\"ahler geometry states that any compact Hermitian manifold with constant Chern holomorphic sectional curvature must be either K\"ahler or Chern flat. The conjecture is known to be true in dimension 2 but still open in dimensions 3 or higher, except for several special classes of Hermitian manifolds. For the important class of locally conformally K\"ahler manifolds, the conjecture was proved by H. Chen, L. Chen, and Nie in 2021 when the constant holomorphic sectional curvature is non-positive and the remaining case was solved recently by Huang and Wan using the result of Kamishima on Bochner-K\"ahler manifolds. In this article, we use their technique to answer similar questions for locally conformally K\"ahler manifolds with constant Levi-Civita or Bismut holomorphic sectional curvature.
\end{abstract}


\maketitle

\markleft{Shuwen Chen and Fangyang Zheng}
\markright{lcK manifolds with constant holomorphic sectional curvature}

\tableofcontents

\section{Introduction and statement of results}\label{intro}

A long-lasting open question in non-K\"ahler geometry is the following:

\begin{conjecture} \label{conj1}
Any compact Hermitian manifold with constant Chern holomorphic sectional curvature must be either K\"ahler or Chern flat.
\end{conjecture}

The conjecture is known to be true in dimension 2 by the work of Balas and Gauduchon \cite{Balas, BG} in 1985 when the constant is zero or negative, and by Apostolov, Davidov, and Muskarov \cite{ADM} in 1996 when the constant is positive. In dimension 3 and higher, the conjecture is still largely open, except in several special cases. For instance, the conjecture was confirmed for all twistor spaces by Davidov, Grantcharov, and Muskarov \cite{DGM}, for all Chern K\"ahler-like manifolds by Tang \cite{Tang}, for all Bismut K\"ahler-like manifolds by Rao and Zheng \cite{RaoZ}, for all compact normal balanced threefolds with non-positive constant Chern holomorphic sectional curvature by Ma and Nie \cite{MaN},  for all complex nilmanifolds by Li and Zheng \cite{LiZ},  and for all non-balanced Bismut torsion-parallel manifolds by Chen and Zheng \cite{ChenZ26}. See \cite{Zheng} for a recent survey on this conjecture.

Locally conformally K\"ahler  manifolds form an important class of  special Hermitian manifolds and there have been extensively studies in the literature. We refer the readers to the recent book by Ornea and Verbitsky \cite{OV} and the references therein for more details. For such  manifolds, the conjecture was proved by H. Chen, L. Chen, and Nie \cite{CCN} in 2021 when the constant is negative or zero. In a recent breakthrough, Huang and Wan \cite{HW} solved the remaining case when the constant is positive, thus establishing the conjecture for all locally conformally K\"ahler manifolds. In summary, their results can be stated as follows

\begin{theorem}[Chen-Chen-Nie \& Huang-Wan, \cite{CCN,HW}] \label{thm0}
Let $(M^n,g)$ be a compact locally conformally K\"ahler manifold with constant Chern holomorphic sectional curvature. Then it must be K\"ahler, thus a complex space form.
\end{theorem}

There are similar  questions to Conjecture \ref{conj1} when the Chern connection is replaced by Levi-Civita or Bismut connection, namely, there are

\begin{conjecture} \label{conj2}
Any compact Hermitian manifold with constant Levi-Civita holomorphic sectional curvature must be either K\"ahler or Levi-Civita flat.
\end{conjecture}

\begin{conjecture} \label{conj3}
Any compact Hermitian manifold with non-zero constant Bismut holomorphic sectional curvature must be K\"ahler.
\end{conjecture}

Note that in the Bismut case, we required the constant to be non-zero in the hypothesis, as there are examples of compact Hermitian manifolds with zero Bismut holomorphic sectional curvature but are not Bismut flat. Generally speaking, Conjectures \ref{conj2} and \ref{conj3} tend to be harder to confirm than Conjecture \ref{conj1}, even though in dimension 2 both of them have been established. For Conjecture \ref{conj2}, the dimension 2 case was proved by Sato and Sekigawa \cite{SS} in 1990 when the constant is negative or zero, and finished by \cite{ADM} when the constant is positive. For Conjecture \ref{conj3}, the dimension 2 case was done by \cite{ChenZ}.

Our main observation here is that, for locally conformally K\"ahler manifolds, the constancy condition on Chern, Levi-Civita, or Bismut holomorphic sectional curvature only differ by a simple term, which does not involve the Bochner part of the curvature decomposition. So the technique of Huang and Wan can be borrowed here to allow the confirmation of Conjectures \ref{conj2} and \ref{conj3} for all compact locally conformally K\"ahler manifolds. The following are the main results of this article:

\begin{theorem}  \label{thm1}
Let $(M^n,g)$ be a compact locally conformally K\"ahler manifold. If its Levi-Civita holomorphic sectional curvature is a constant, then it must be K\"ahler, thus a complex space form.
\end{theorem}

\begin{theorem}  \label{thm2}
Let $(M^n,g)$ be a compact locally conformally K\"ahler manifold. If its Bismut holomorphic sectional curvature is a constant $c$, then either $g$ is K\"ahler, or $c=0$ and $(M^n,g)$ is an isosceles Hopf manifold.
\end{theorem}

Recall that an isosceles Hopf manifold is a compact Hermitian manifold $(M^n,g)$ such that a finite unbranched cover of it is the quotient of the punctured complex Euclidean space by the infinite cyclic group generated by a holomorphic contraction $\gamma$, in the form
$$ \big({\mathbb C}^n\setminus \{0\}\big)/ \langle \gamma \rangle , \ \ \ \gamma(z_1,  \ldots , z_n) = (a_1z_1,  \ldots , a_nz_n), $$
where $|a_1|= \cdots = |a_n| \in (0,1)$, equipped with the metric whose K\"ahler form is
$$\omega = \frac{\alpha }{|z|^2} \sqrt{-1} \,\partial \overline{\partial} \,|z|^2, $$
where $\alpha >0$ is a constant and $|z|^2 = |z_1|^2 + \cdots + |z_n|^2$. In other words, an isosceles Hopf manifold is a special kind of diagonalizable linear Hopf manifold where the standard metric conformal to the Euclidean metric descends. Such a manifold has vanishing Bismut holomorphic sectional curvature, and it is Bismut flat when and only when $n=2$.

As is well-known, complete K\"ahler manifolds with constant holomorphic sectional curvature are called {\em complex space forms.} Up to multiplying the metric by a positive constant, their universal covers are either ${\mathbb C}{\mathbb P}^n$, ${\mathbb C}^n$, or ${\mathbb C}{\mathbb H}^n$, equipped with the standard metrics. The point in Conjectures \ref{conj1} through \ref{conj3} is to understand the Hermitian version of the space forms, and in the case of Chern or Levi-Civita connection, the conjectured conclusion is that the metric must be either K\"ahler (hence a complex space form) or flat. The Bismut connection case is slightly more complicated as there are examples of manifolds with vanishing Bismut holomorphic sectional curvature but are not Bismut flat, even though such manifolds seem to be highly restrictive and their classification or characterization might be an interesting problem itself.

For the flat cases, we recall that compact Chern flat manifolds are exactly compact quotients of complex Lie groups equipped with compatible left-invariant Hermitian metrics, by the classic result of Boothby \cite{Boothby}, while compact Bismut flat manifolds are exactly the compact quotients of Samelson spaces, which are simply-connected Lie groups with bi-invariant metrics and compatible left-invariant complex structures. Such spaces are known to be Bismut flat \cite{Pittie} and the converse was due to Wang, Yang, and Zheng \cite{WYZ} in 2020.

Compact Levi-Civita flat Hermitian manifolds are lesser known. By the classic Bieberbach theorem, we know that such manifolds (up to a finite unbranched cover) must be a flat torus $T^{2n}_{\mathbb R}$ equipped with a compatible complex structure. Such a complex structure must be constant when $n\leq 2$, but there are non-constant (hence non-K\"ahlerian) ones when $n\geq 3$. In \cite{KYZ}, Khan, Yang, and Zheng classified all such spaces in complex dimension $3$, but for $n\geq 4$ this is still an open question, even though there are lots of non-K\"ahlerian examples.

The proofs of Theorems \ref{thm1} and \ref{thm2} follow the technique and approach of Huang and Wan in \cite{HW}, and this article is organized as follows. In the next two sections, we will set up the notations and recall some known results. In \S 4 we will give the curvature calculation for Hermitian metrics under conformal changes. In \S 5, we will deal with the globally conformally K\"ahler case, and leave the strictly  locally conformally K\"ahler case to \S 6.

\vspace{0.3cm}

\section{Holomorphic sectional curvatures}

In this section, we will recall the notion of holomorphic sectional curvature for the three canonical metric connections of Hermitian manifolds. Let $(M^n,g)$ be a Hermitian manifold of complex dimension $n\geq 2$. Suppose $D$ is a metric connection on $M^n$, namely, $D$ satisfies $Dg=0$. Its torsion and curvature are defined by
$$ T^D(x,y)=D_xy-D_yx-[x,y], \ \ \ \ \ \ R^D_{xy}z=D_xD_yz-D_yD_xz -D_{[x,y]}z, $$
for any vector fields $x$, $y$, $z$ on $M$. Using the metric, we will also write $R^D(x,y,z,w)$ or the abbreviation $R^D_{xyzw}$ for $\langle R^D_{xy}z, w\rangle$, where $g=\langle , \rangle$ is the metric, extended as a complex bilinear form. Note that since $D$ is metric, the $4$-tensor $R^D$ is skew-symmetric with respect to its first two positions as well as its last two positions, so one can define the {\em sectional curvature} of $D$ for any $2$-plane $\pi \subset T_pM$ contained the tangent space of $M$:
\begin{equation*}
K^D(\pi ) = R^D(x,y,y,x)/ |x\wedge y|^2,
\end{equation*}
where $\{x,y\}$ is any basis of $\pi$, and $|x\wedge y|^2 =|x|^2|y|^2 - \langle x,y\rangle^2$. Clearly, $K^D(\pi)$ is independent of the choice of the basis. In particular, when $\pi$ is $J$-invariant, one can choose a basis of $\pi$ in the form $\{ x,Jx\}$ and get the  {\em holomorphic sectional curvature} of $D$:
\begin{equation}
H^D(X) = R^D(x,Jx,Jx,x)/ |x|^4 = R^D(X,\overline{X}, X, \overline{X})/|X|^4,
\end{equation}
where $X=x-\sqrt{-1}Jx$ is the corresponding type $(1,0)$ complex tangent vector. The following fact is well-known so we omit the proof here:
\begin{lemma}\label{lemma1}
Let \(D\) be a  metric connection on a Hermitian manifold \((M^n,g)\).  Then
we have
\begin{equation}\label{eq:H=c}
H^D\equiv c   \ \ \ \Longleftrightarrow \ \ \ \widehat R^D_{i\bar j k\bar\ell}
 =\frac{c}{2}
 \left(g_{i\bar{j}}g_{k\bar{\ell}}
      +g_{i\bar{\ell}}g_{k\bar{j}}\right), \ \ \forall \ 1\leq i,j,k,\ell \leq n,
\end{equation}
under any local frame $e=\{ e_1, \ldots , e_n\}$ of type $(1,0)$ complex vector fields on $M^n$. Here we used
\begin{equation}\label{eq:sym}
 \widehat P_{i\bar j k\bar\ell}
 =\frac14\left(
 P_{i\bar j k\bar\ell}
 +P_{k\bar j i\bar\ell}
 +P_{i\bar\ell k\bar j}
 +P_{k\bar\ell i\bar j}\right)
\end{equation}
to denote the symmetrization of a tensor $P$.
\end{lemma}

Given a Hermitian manifold $(M^n,g)$, we will denote by $\nabla$, $\nabla^g$, $\nabla^b$ the Chern, Levi-Civita, and Bismut connection, respectively. Denote by $R$, $R^g$, $R^b$ the corresponding curvature tensors. The Chern, Levi-Civita, and Bismut holomorphic sectional curvatures are denoted by $H$, $H^g$, and $H^b$. Also let us write $T$ for the torsion tensor of the Chern connection $\nabla$. Under any local frame $e$ of type $(1,0)$ vector fields, write
\begin{equation}
T(e_i,e_k) = \sum_j T^j_{ik} e_j, \ \ \ \  \ \ \forall \ 1\leq i,k\leq n,
\end{equation}
for the Chern torsion components under $e$. We have $T^j_{ik}=-T^j_{ki}$ for any $i,j,k$. As is well-known, one always has $T(e_i, \overline{e}_k)=0$.
Now suppose that the local frame  $e=\{ e_1, \ldots , e_n\}$ is unitary, namely, $\langle e_i, \overline{e}_j\rangle = \delta_{ij}$ for any $1\leq i,j\leq n$. Then the relationships between the three curvature tensors are given by

\begin{lemma}
Let $e$ be a local unitary frame on the Hermitian manifold $(M^n,g)$. Then the Bismut and Levi-Civita curvatures are related to the Chern curvature by
\begin{eqnarray}
R^b_{i\bar{j}k\bar{\ell}} & = & R_{i\bar{j}k\bar{\ell}} +T^{\ell}_{ik,\bar{j}} + \overline{ T^k_{j\ell ,\bar{i}} } +\sum_s \big\{ T^s_{ik} \overline{ T^s_{j\ell }} - T^{\ell}_{is} \overline{ T^k_{js }} \big\},   \label{eq:RbR}\\
R^g_{i\bar{j}k\bar{\ell}} & = & R_{i\bar{j}k\bar{\ell}}  + \frac{1}{2} \big( T^{\ell}_{ik,\bar{j}} + \overline{ T^k_{j\ell ,\bar{i}} } \big)  +\frac{1}{4}\sum_s \big\{ T^s_{ik} \overline{ T^s_{j\ell }} - T^{\ell}_{is} \overline{ T^k_{js }} - T^{j}_{ks} \overline{ T^i_{\ell s }} \big\}, \label{eq:RgR}
\end{eqnarray}
for any $1\leq i,j,k,\ell \leq n$, where indices after comma stand for covariant derivatives with respect to the Chern connection $\nabla$.
\end{lemma}

\begin{proof}
Denote by $\varphi$ the coframe dual to $e$, and write $e$ and $\varphi$ as column vectors. Under the frame $e$, denote by $\theta$, $\theta^b$, $\Theta$, $\Theta^b$ the matrices of connection and curvature for the Chern and Bismut connection, respectively. So in matrix form we have
$$\nabla e = \theta \otimes  e, \ \ \ \ \ \nabla^b e = \theta^b \otimes  e, \ \ \ \ \ \Theta = d\theta - \theta \wedge \theta,  \ \ \ \ \ \Theta^b = d\theta^b - \theta^b \wedge \theta^b. $$
Similarly, for the Levi-Civita connection $\nabla^g$ let us write
$$ \nabla^g e = \theta_1 \otimes e + \overline{\theta}_2 \otimes \overline{e}, \ \ \ \ \ \  \Theta_1 = d\theta_1 - \theta_1 \wedge \theta_1 - \overline{\theta}_2 \wedge \theta_2. $$
Then the curvature components of these connections are given by
$$ R_{i\bar{j}k\bar{\ell}}= \Theta_{k\ell} (e_i, \overline{e}_j), \ \ \ \ \ R^b_{i\bar{j}k\bar{\ell}}= \Theta^b_{k\ell} (e_i, \overline{e}_j), \ \ \ \ \ R^g_{i\bar{j}k\bar{\ell}}= (\Theta_1)_{k\ell} (e_i, \overline{e}_j). $$
Write $\gamma=\theta^b-\theta $, then it is  well-known (see for example \cite{YangZ}) that
$$ \gamma_{ij} = \sum_k \big\{ T^j_{ik}\varphi_k - \overline{T^i_{jk} } \overline{\varphi}_k \big\}, \ \ \ \ \theta_1 =\theta + \frac12 \gamma, \ \ \ \ \ (\theta_2)_{ij} = \frac12 \sum_k \overline{T^k_{ij}} \varphi_k, $$
for any $1\leq i,j\leq n$. The structure equation is $d\varphi = - \,^t\!\theta \wedge \varphi + \tau$, where $\tau_j = \frac12 \sum_{i,k}T^j_{ik}\varphi_i\wedge \varphi_k$. From these, a straight-forward computation leads to (\ref{eq:RbR}) and (\ref{eq:RgR}).
\end{proof}

By taking the symmetrization (\ref{eq:sym}) in (\ref{eq:RbR}) and (\ref{eq:RgR}) and using the skew-symmetry of torsion: $T^j_{ik}=-T^j_{ki}$, we get the following
\begin{corollary} \label{cor3}
Under any local unitary frame $e$ of a Hermitian manifold $(M^n,g)$, it holds that
\begin{equation}  \label{eq:Q}
\left\{ \begin{split} \widehat{R^b} = \widehat{R} -Q,    \ \ \ \ \ \widehat{R^g} = \widehat{R} -\frac12 Q, \hspace{2.3cm}\\
  Q = \frac14 \sum_s \big( T^{j}_{is} \overline{ T^k_{\ell s }} + T^{\ell }_{ks} \overline{ T^i_{j s }} + T^{j}_{ks} \overline{ T^i_{\ell s }} + T^{\ell}_{is} \overline{ T^k_{j s }} \big).
  \end{split} \right.
\end{equation}
\end{corollary}

\vspace{0.3cm}

\section{The map $L$ and Bochner-K\"ahler manifolds}

In this section we will recall the basics about Bochner-K\"ahler manifolds. Given a complex manifold $M^n$, let us denote by ${\mathcal H}$ the space of all Hermitian symmetric $(1,1)$ tensors on $M^n$. In a local coordinate neighborhood, an element $P \in {\mathcal H}$ can be expressed as
$$ P = \sum_{i,j=1}^n P_{i\bar{j}} \,dz_i \otimes d\overline{z}_j, \ \ \ \ \ \ P_{j\bar{i}} = \overline{P_{i\bar{j}}}.  $$
The tensor is associated to a real $(1,1)$-form on $M^n$
$$ \phi_P = \sqrt{-1}\sum_{i,j=1}^n P_{i\bar{j}} \,dz_i \wedge d\overline{z}_j,$$
so ${\mathcal H}$ is isomorphic to the space ${\mathcal A}^{1,1}_{\mathbb R}$ of all real $(1,1)$-forms on $M^n$. Now let $(M^n,g)$ be a Hermitian manifold. The Hermitian metric $g$ defines a linear map from ${\mathcal H}$ into the space of $4$-tensors on $M^n$ of type $(2,2)$:
\begin{equation} \label{eq:Ldef}
L_g (P) = P_{i\bar{j}} g_{k\bar{\ell}} +  P_{k\bar{\ell}} g_{i\bar{j}} + P_{i\bar{\ell}} g_{k\bar{j}} + P_{k\bar{j}} g_{i\bar{\ell}}, \ \ \ \ \ 1\leq i,j,k,\ell \leq n,
\end{equation}
where $P_{i\bar{j}}$ and $g_{i\bar{j}}$ are components of $P$ and $g$ under any local frame of type $(1,0)$ tangent vectors. Clearly, for any $P\in {\mathcal H}$ the tensor $L=L_g(P)$ obeys the K\"ahler symmetry
\begin{equation}
 L_{k\bar{j}i\bar{\ell}} = L_{i\bar{j}k\bar{\ell}} , \ \ \ \ \overline{L_{j\bar{i}\ell \bar{k}}} =L_{i\bar{j}k\bar{\ell}}, \ \ \ \ \ \ \forall \ 1\leq i,j,k,\ell \leq n,
 \end{equation}
and we always have
\begin{equation}
L_g(g)=2G_g, \ \ \ \ \ \mbox{where} \ \ G_g = g_{i\bar{j}} g_{k\bar{\ell}} + g_{i\bar{\ell}} g_{k\bar{j}}.
 \end{equation}
The following lemma is due to Huang and Wan \cite[Lemma 3.1]{HW}, which says that the map $L_g$ is always injective.
\begin{lemma}[\cite{HW}] \label{lemma-injective}
Let $(M^n,g)$ be a Hermitian manifold. If $P\in {\mathcal H}$ such that $L_g(P)=0$, then $P=0$. \end{lemma}

\begin{proof}
Denote by $(g^{\bar{j}i})$ the inverse matrix of $(g_{i\bar{j}})$ under any local frame $e$ of type $(1,0)$ vectors. Then we have
$$ \sum_{k,\ell} g^{\bar{\ell}k} \big(L_g(P)\big)_{i\bar{j}k\bar{\ell}} = (n+2)P_{i\bar{j}} + \mbox{tr}_g(P)\,g_{i\bar{j}}, \ \ \,  \sum_{i,j,k,\ell} g^{\bar{j}i}g^{\bar{\ell}k} \big(L_g(P)\big)_{i\bar{j}k\bar{\ell}} = (2n+2)\,\mbox{tr}_g(P).$$
So when $L_g(P)=0$, we have $\mbox{tr}_g(P)=0$ by the second equality in the above line. Plug it into the first equality, we get $P=0$.
\end{proof}

Now suppose that $(M^n,g)$ is a K\"ahler manifold. The K\"ahlerness implies that the Chern, Levi-Civita, and Bismut connections all coincide.
Let $e=\{ e_1, \ldots , e_n\}$ be a local frame of type $(1,0)$ complex tangent vector fields in $M^n$. Denote by $g_{i\bar{j}}=g(e_i, \overline{e}_j)$ and $R_{i\bar{j}k\bar{\ell}}$ the components of the metric and the curvature under the frame. The {\em traceless Ricci} tensor of $g$ is defined by
$$  r_{i\bar{j}} = Ric_{i\bar{j}} - \frac{s}{n}g_{i\bar{j}}, $$
where
$$ Ric_{i\bar{j}} = \sum_{k,\ell} g^{\bar{\ell }k} R_{i\bar{j}k\bar{\ell}}, \ \ \ s=\sum_{i,j} g^{\bar{j}i} Ric_{i\bar{j}}. $$
Note that $s$ is the trace of the Ricci tensor $Ric$, and it equals to half of the Riemannian scalar curvature.
As it is well-known, the curvature tensor of the K\"ahler metric $g$ can be decomposed in the following way
\begin{equation} \label{eq:decom}
R_{i\bar{j}k\bar{\ell}} =  \frac{s}{n(n+1)} (G_g)_{i\bar{j}k\bar{\ell}} + \frac{1}{n+2} (L_g(r))_{i\bar{j}k\bar{\ell}} + B_{i\bar{j}k\bar{\ell}} ,
\end{equation}
In other words, one can define a $4$-tensor $B$ by the above equation, which is called the {\em Bochner curvature} of $g$, and $g$ is said to be {\em Bochner-K\"ahler} or {\em Bochner-flat} if $B=0$. From the algebraic point of view, it is relatively easy to detect if a given K\"ahler metric will be Bochner-K\"ahler or not by the following lemma, which is  well-known and we include the proof here for readers' convenience.

\begin{lemma} \label{lemma3}
A K\"ahler manifold $(M^n,g)$ is Bochner-K\"ahler if and only if there exists a Hermitian symmetric $(1,1)$ tensor $P\in {\mathcal H}$ on $M^n$ such that $R=L_g(P)$.
\end{lemma}

 \begin{proof}
If $g$ is Bochner-K\"ahler, then $B=0$, so by (\ref{eq:decom}) and $L_g(g)=2G_g$ we get
\begin{equation} \label{eq:temp}
 R = L_g\big(\frac{s}{2n(n+1)}g + \frac{1}{n+2}r\big).
 \end{equation}
Conversely, if $R=L_g(P)$ for some $P\in {\mathcal H}$, then we have
$$ Ric = (n+2)P + \mbox{tr}_g(P)\,g, \ \ \ s =2(n+1) \,\mbox{tr}_g(P).$$
From this, we get
$$ \frac{1}{n+2}r= P - \frac{1}{n}\mbox{tr}_g(P)\,g , \ \ \ \frac{s}{n(n+1)} = \frac{2}{n} \,\mbox{tr}_g(P).  $$
Hence (\ref{eq:temp}) holds, and $B=0$ by (\ref{eq:decom}), so $g$ is Bochner-K\"ahler by definition. This completes the proof of the lemma.
 \end{proof}

Denote by $M^n_{\kappa}$  the complete, simply-connected K\"ahler manifold with constant holomorphic sectional curvature $\kappa$. By a theorem of Kamishima \cite{Kamishima1994} (see also the excellently written paper by Bryant \cite{Bryant} which gives a systematic treatment on Bochner-K\"ahler metrics), complete Bochner-K\"ahler manifolds are always symmetric:
\begin{theorem}[\cite{Kamishima1994, Bryant}] \label{thm4}
Let $(M^n,g)$ be a compact Bochner-K\"ahler manifold. Then its universal covering space is holomphically isometric to $M^p_{\kappa}\times M^{n-p}_{-\kappa}$ for some constant $\kappa$ and some  integer $p$ with $1\leq p\leq n$. In particular, $g$ is locally symmetric.
\end{theorem}

\vspace{0.3cm}

\section{Conformal change of Hermitian metrics}

First let us recall some basic relations between the torsion and curvature of conformally related Hermitian metrics. Suppose that $(M^n,g)$ is a Hermitian manifold, and $h =e^{2u}g$ where $u \in C^{\infty}(M)$ is a real-valued smooth function. Denote by $T$ and $R$ the torsion and curvature of the Chern connection of $g$, and by $T^h$, $R^h$ the torsion and curvature of the Chern connection of $h$. Let $e$ be a local unitary frame for $g$ with dual coframe $\varphi$. As before, write
$$T(e_i, e_k) = \sum_j T^j_{ik}e_j, \ \ \ \mbox{and} \ \  R_{i\bar{j}k\bar{\ell}}=R(e_i,\overline{e}_j, e_k, \overline{e}_{\ell}) $$
for the components of $T$ and $R$ under $e$. For each $1\leq i\leq n$,   write $\varepsilon_i=e^{-u}e_i$, $\psi_i=e^u\varphi_i$. Then $\varepsilon$ and $\psi$ are local unitary frame and dual coframe for $h$. Let us denote by $(T^h)^j_{ik}$, $R^h_{i\bar{j}k\bar{\ell}}$ the torsion and curvature components for the Chern connection of $h$ under the unitary frame $\varepsilon$, then we have

\begin{lemma}
Under any unitary frame $e$ for $g$ and the corresponding unitary frame $\varepsilon=e^{-u}e$ for the conformal metric $h=e^{2u}g$, their Chern torsion and curvature components are related by
\begin{eqnarray}
&& R^h_{i\bar{j}k\bar{\ell}} = e^{-2u} \big( R_{i\bar{j}k\bar{\ell}} - 2u_{i\bar{j}}  \delta_{k\ell} \big),  \label{eq:RhR}\\
&& (T^h)^j_{ik} = e^{-u} \big( T^j_{ik} + 2u_i\delta_{jk} -2 u_k \delta_{ji} \big),  \label{eq:ThT}
\end{eqnarray}
for any $1\leq i,j,k, \ell \leq n$. Here $u_i= (\partial u)(e_i)=e_i(u)$, and $u_{i\bar{j}}= (\partial \overline{\partial } u)(e_i, \overline{e}_j)$.
\end{lemma}

\begin{proof}
This is a straightforward computation. See for instance the discussion in \S 5 of \cite{YangZ}.
\end{proof}

In particular, if in the above the metric $h$ is K\"ahler, then $T^h=0$, so by (\ref{eq:ThT}) we have
\begin{equation}  \label{eq:T2}
T^j_{ik} =  2 u_k \delta_{ji} -2u_i\delta_{jk},
\end{equation}
then by (\ref{eq:Q}) and a straightforward computation we get
\begin{equation}  \label{eq:Q2}
Q =  |\nabla^g u|^2 \,G - L_g(\xi), \ \ \ \ \ \ \xi_{i\bar{j}} = u_i \overline{u_j} =(\partial u\wedge \overline{\partial}u)(e_i,\overline{e}_j).
\end{equation}
Here $\nabla^g$ is the Levi-Civita connection of $g$, and $|\nabla^g u|^2=2\sum_{s=1}^n |u_s|^2$.

An important observation in \cite{HW} is that, if a Hermitian metric $g$ is  conformal to a  K\"ahler metric $h$, and $g$ has constant Chern holomorphic sectional curvature, then $h$ is Bochner-K\"ahler. Our main  observation in this article is the following analogous statement:

\begin{proposition}  \label{prop1}
Let $(M^n,g)$ be a Hermitian manifold, and $h=e^{2u}g$ where $u\in C^{\infty}(M)$ is a real-valued smooth function. If $h$ is K\"ahler, and $g$ has constant Chern (or Levi-Civita, or Bismut) holomorphic sectional curvature, then $h$ is Bochner-K\"ahler.
\end{proposition}

\begin{proof}
If $g$ has constant Chern holomorphic sectional curvature, $H_g=c$, then by (\ref{eq:H=c}) in Lemma \ref{lemma1}, we have $\widehat{R}=\frac{c}{2}G_g$. Since $h$ is K\"ahler, its curvature $R^h$ will obey all K\"ahler symmetries, so $R^h=\widehat{R^h}$. Now by taking the symmetrization in (\ref{eq:RhR}),  we obtain
$$ R^h_{i\bar{j}k\bar{\ell}} = e^{-2u}\frac{c}{2}(\delta_{ij}\delta_{k\ell}+\delta_{i\ell}\delta_{kj} ) -\frac12 \big( \phi_{i\bar{j}} \delta_{k\ell} + \phi_{k\bar{\ell}} \delta_{ij} + \phi_{i\bar{\ell }} \delta_{kj} + \phi_{k\bar{j}} \delta_{i\ell} \big),  $$
where $\phi_{i\bar{j}} = (\partial \overline{\partial}u)(\varepsilon_i, \overline{\varepsilon}_j)$, so $\phi$ is the Hermitian symmetric $(1,1)$ tensor corresponding to the real $(1,1)$-form $\sqrt{-1}\partial \overline{\partial}u$. That is,
\begin{equation} \label{eq:Ac}
H_g=c \ \ \ \Longrightarrow \ \ \ R^h=L_h(A), \ \ A=\frac{c}{4}e^{-2u}h - \frac12 \phi.
\end{equation}
In particular, by Lemma \ref{lemma3} we know that $h$ is Bochner-K\"ahler.

Next let us assume that $g$ has constant Bismut holomorphic sectional curvature, namely, $H^b_g=c$. Then by Lemma \ref{lemma1} we have $\widehat{R^b}=\frac{c}{2}G_g$. By (\ref{eq:Q}) in Corollary \ref{cor3} and (\ref{eq:Q2}), we obtain
$$ \widehat{R} = \widehat{R^b} + Q = \frac{c}{2}G + |\nabla^g u|^2 G - L_g(\xi), $$
where $\xi$ is the Hermitian symmetric $(1,1)$ tensor corresponding to  $\sqrt{-1}\partial u\wedge \overline{\partial}u$. By taking the symmetrization in (\ref{eq:RhR}) again  we obtain
\begin{equation}  \label{eq:Ab}
H_g^b=c \ \ \ \Longrightarrow \ \ \ R^h=L_h(A^b), \ \ A^b=\big(\frac{c}{4}e^{-2u}+\frac12 |\nabla^hu|^2\big)h - \frac12 \phi -\xi.
\end{equation}
Again by  Lemma \ref{lemma3} we know that $h$ is Bochner-K\"ahler in this case. Similarly, if the metric $g$ has constant Levi-Civita holomorphic sectional curvature $H_g^g=c$, then the same argument indicates that $h$ is Bochner-K\"ahler, and in this case we have
\begin{equation} \label{eq:Ag}
H_g^g=c \ \ \ \Longrightarrow \ \ \ R^h=L_h(A^g), \ \ A^g=\big(\frac{c}{4}e^{-2u}+\frac14 |\nabla^hu|^2\big)h - \frac12 \phi -\frac12 \xi.
\end{equation}
This completes the proof of Proposition \ref{prop1}.
\end{proof}

On the K\"ahler manifold $(M^n,h)$, let us denote by $\chi_{\beta}$ the complex Hessian of a function $\beta$. In the Bismut case, let us take $f=e^{2u}$. Because
$$ \partial \overline{\partial}f = 2e^{2u}\partial \overline{\partial}u + 4e^{2u}\partial u\wedge \overline{\partial}u, $$
 we get $\chi_f = 4f (\frac12 \phi + \xi)$. Also, $|\nabla^hu|^2 = 2|\partial u|^2 = \frac{1}{2f^2}|\partial f|^2$, therefore we obtain
\begin{equation} \label{eq:Ab2}
4fA^b = \alpha h - \chi_f,  \ \ \ \ \alpha = c+\frac{1}{f}|\partial f|^2.
\end{equation}

Similarly, in the Levi-Civita case, by letting $w=e^u$, we have $\chi_w = w(\phi + \xi)$, and

\begin{equation} \label{eq:Ag2}
2wA^g = \beta h - \chi_w,  \ \ \ \ \beta = \frac{c}{2w}+\frac{1}{w}|\partial w|^2.
\end{equation}

\vspace{0.3cm}

\section{Globally conformally K\"ahler manifolds}

In this section, let us prove Theorems \ref{thm1} and \ref{thm2} in the case when the  manifold is globally conformally K\"ahler, as a generalization to Proposition 4.1 of \cite{HW}.

\begin{proposition} \label{prop2}
Let $(M^n,g)$ be a compact Hermitian manifold, such that $h=e^{2u}g$ is K\"ahler for some smooth real-valued function $u\in C^{\infty}(M)$. If $g$ has constant Chern or Bismut or Levi-Civita holomorphic sectional curvature. Then $g$ is K\"ahler, hence $(M^n,g)$ is a complex space form.
\end{proposition}

\begin{proof}
The Chern case is due to Huang and Wan in \cite[Proposition 4.1]{HW}. First let us assume that the metric $g$ has constant Bismut holomorphic sectional curvature, namely, $H^b_g=c$. Since $h=e^{2u}g$ is K\"ahler, we can apply Proposition \ref{prop1} to conclude that $(M^n,h)$ is Bochner-K\"ahler, hence by Theorem \ref{thm4} we know that $(M^n,h)$ is locally Hermitian symmetric, namely, $\nabla^hR^h=0$.  By (\ref{eq:Ab}), we know that for any vector field $x$ on $M$ we have
$$ 0 = \nabla^h_xR^h = \nabla^h_xL_h(A^b) = L_h(\nabla^h_xA^b), $$
so by the injectivity of $L_h$, Lemma \ref{lemma-injective}, we get $\nabla^h_xA^b=0$ for any $x$, hence the Hermitian symmetric $(1,1)$ tensor $A^b$ is  parallel on the compact K\"ahler manifold $(M^n,h)$.

\vspace{0.1cm}

\noindent {\bf Claim.} $A^b$ either has one eigenvalue or has exactly two non-zero eigenvalues whose sum is zero.

\vspace{0.1cm}

To see this, notice that as the tensor $A^b$ is parallel, its eigenspaces split $(M^n,h)$ into local product of K\"ahler manifolds. If $A^b$ has two distinct eigenvalues, say $\lambda \neq \mu$. Take unit eigenvector $X$, $Y$ of type $(1,0)$ corresponding to $\lambda$ and $\mu$, respectively. Then the mixed curvature will vanish: $R^h_{X\overline{X}Y\overline{Y}}=0$. But $R^h=L_h(A^b)$, we have
$$ 0 = R^h_{X\overline{X}Y\overline{Y}} = \big(L_h(A^b)\big)_{X\overline{X}Y\overline{Y}} = \lambda + \mu.$$
This shows that either $A^b$ is a constant multiple of $h$ or it has exactly two eigenvalues $\lambda$ and $-\lambda$, where $\lambda >0$ is a constant. So the claim is proved, and in the following we will divide the proof into two cases.

\vspace{0.1cm}

\noindent {\bf Case 1.} $A^b=ah$ for some constant $a$.

In this case, by (\ref{eq:Ab2}), we get $\chi_f = (\alpha -4af)h$. Write $F=\alpha-4af$, then this means that
$$ \sqrt{-1}\partial \overline{\partial } f = F \omega_h, $$
where $\omega_h$ is the K\"ahler form of $h$. Since $d\omega_h=0$, by taking the differential on the above line, we get $dF \wedge \omega_h=0$ which implies $dF=0$, hence $F$ must be a constant. On the other hand, by taking trace on the above line, we get $nF=\Delta f$, its integral over the compact manifold $M^n$ is zero, so $F$ is zero hence $f$ is a constant. Therefore the metric $g=e^{-2u}h=f^{-1}h$ is K\"ahler.

\vspace{0.1cm}

\noindent {\bf Case 2.} $A^b$ has exactly two eigenvalues $\lambda$ and $-\lambda$, where $\lambda >0$.

In this case, everywhere on $M^n$ the holomorphic tangent space splits as the direct sum of the two eigenspaces of $A^b$: $T^{1,0}M=E_{\lambda} \oplus E_{-\lambda}$. Again by (\ref{eq:Ab2}) we have
$$ \chi_f = \alpha h - 4fA^b, \ \ \ \ \alpha = c+\frac{1}{f}|\partial f|^2. $$
Let $p\in M$ be a point where $f$ reaches its maximum. Take a unit vector $X\in E_{-\lambda}(p)$. Then since the gradient of $f$ vanishes at $p$ and the complex Hessian of $f$ is non-positive, we get
$$ 0 \geq \chi_f(X,\overline{X}) = \alpha (p) -4f(p)(-\lambda) = c+4f(p)\lambda, $$
hence $c\leq -4f(p)\lambda <0$ as $\lambda >0$ and $f$ is positive. On the other hand, let $q$ be a minimum point of $f$ and take a unit vector $Y\in E_{\lambda }(q)$, then as the complex Hessian is non-negative at $q$, we get
$$ 0 \leq \chi_f(Y,\overline{Y}) = \alpha (q) - 4f(q)\lambda = c-4f(q)\lambda, $$
hence $c\geq 4f(q)\lambda >0$, contradicting to our previous conclusion that $c<0$. Therefore, this case cannot occur. We have thus completed the proof of Proposition \ref{prop2} for the Bismut case.

The Levi-Civita case goes similarly, by using (\ref{eq:Ag2}) and the complex Hessian of $w$ instead. Since the proof is exactly analogous, we will omit it here.
\end{proof}

\vspace{0.3cm}

\section{Locally conformally K\"ahler manifolds}

In this section, we will deal with the remaining part in the proofs of Theorems \ref{thm1} and \ref{thm2}.  First of all let us recall that a {\em locally conformally K\"ahler manifold} is a Hermitian manifold whose K\"ahler form $\omega$ satisfies
\begin{equation}
 d\omega = \theta \wedge \omega ,
\end{equation}
where $\theta$ is a closed $1$-form on $M^n$, called the {\em Lee form} of $(M^n,g)$. Clearly this $1$-form (when exists) is uniquely determined. Equivalently, on the universal covering space $\pi : \widetilde{M} \rightarrow M$, the lifted metric $\widetilde{g}:=\pi^{\ast}\!g$ is conformal to a K\"ahler metric $\widetilde h=e^{2u}\widetilde{g}$, where $\pi^{\ast}\theta = -2du$. Note that in general the metric $\widetilde h$ on $\widetilde{M}$ may not be complete.

Given a compact locally conformally K\"ahler manifold $(M^n,g)$, either $\theta$ is exact, in which case $g$ is globally conformally K\"ahler, or $\theta$ is not exact, in which case $(M^n,g)$ or any of its finite unbranched cover is not globally conformally K\"ahler. The latter case will be called {\em essential} or {\em strict} locally conformally K\"ahler. In this case $M^n$ is actually non-K\"ahlerian, namely, it admits no K\"ahler metric.

Throughout this section, $(M^n,g)$ will be assumed to be a compact locally conformally K\"ahler manifold which is strict. Denote by $\widetilde{g}$ the lifted (or pull-back) metric of $g$ on the universal covering space $\pi: \widetilde{M} \rightarrow M$. Let $u$ be a real-valued smooth function on $\widetilde{M}$ such that $\pi^{\ast}\theta = -2du$. Note that such a function is unique up to additive constants. Then $\widetilde{h}=e^{2u}\widetilde{g}$ is a K\"ahler metric on $\widetilde{M}$, which is not necessarily complete.

Assume that $g$ has constant Chern, Levi-Civita, or Bismut holomorphic sectional curvature. Then $\widetilde{g}$ will also have the same property, so by Proposition \ref{prop1}, the metric $\widetilde{h}$ will be Bochner-K\"ahler. As the metric $\widetilde{h}$ may not be complete, we do not know if it is symmetric any more. In this case, the result of Huang and Wan in \cite[Proposition 5.1]{HW} becomes a crucial tool, which stems from Kamishima's uniformization theorem \cite{Kamishima1994, Kamishima2005, Kamishima2006} and Fried's classification on compact similarity manifolds \cite{Fried1980}.

\begin{theorem}[Kamishima--Fried uniformization \cite{HW}]\label{thm:uniformization}
Let \((M^n,g)\) be a compact locally conformally K\"ahler manifold with \(n\geq2\) such that its Lee form is not exact. Assume that the associated K\"ahler metric $\widetilde{h}$ is Bochner-K\"ahler.  Then up to scaling $\widetilde{h}$ by a positive constant, there is a holomorphic
isometry
\[
 (\widetilde M,\widetilde h)
 \cong(\C^n\setminus\{0\},h_{0}),
\]
where $h_0$ is the standard Euclidean metric. Under this identification every deck transformation has the form
\begin{equation} \label{eq:deck}
 \gamma(z)=r_\gamma U_\gamma z,
 \qquad r_\gamma>0,\   U_\gamma\in U(n), \   \mbox{and} \ r_\gamma\neq1 \ \mbox{for at least one} \  \gamma .
\end{equation}
\end{theorem}

Now we are ready to finish the proof of Theorems \ref{thm1} and \ref{thm2}.

\vspace{0.05cm}

\begin{proof}[{\bf Proof of Theorem \ref{thm1}.}]
Let $(M^n,g)$ be a compact locally conformally K\"ahler manifold such that its Levi-Civita holomorphic sectional curvature is constant: $H^g=c$. If $g$ is globally conformally K\"ahler, then by Proposition \ref{prop2} we know that $g$ must be K\"ahler, hence is a complex space form. So to prove Theorem \ref{thm1}, it suffices to rule out the strict case. Suppose that $g$ is not globally conformally K\"ahler, namely, the Lee form is not exact on $M^n$. Write $\pi: \widetilde{M} \rightarrow M$ for the universal cover, $\widetilde{g}=\pi^{\ast}\!g$ the pull-back, and $\widetilde{h}=e^{2u}\widetilde{g}$ the K\"ahler metric on $\widetilde{M}$. Then by Proposition \ref{prop1} we know that $\widetilde{h}$ is Bochner-K\"ahler, hence by scaling the metric $g$ by a suitable constant multiple if necessary, Theorem \ref{thm:uniformization} tells us that $(\widetilde M, \widetilde h)$ is holomorphically isometric to $({\mathbb C}^n\setminus \{ 0\}, h_0)$ where $h_0$ is the standard Euclidean metric, and under this identification, each deck transformation $\gamma$ of $\pi$ is in the form of (\ref{eq:deck}). Since the Levi-Civita holomorphic sectional curvature of $\widetilde g$ is a constant $c$, we have formula (\ref{eq:Ag}) and (\ref{eq:Ag2}), where $g$ and $h$ are replaced by $\widetilde g$ and $\widetilde h$. Since $\widetilde h$ can be identified with $h_0$ and is flat, we get $A^g=0$ by the injectivity of $L_{\widetilde{h}}$, that is, we have
$$ \beta \widetilde h - \chi_w =0, \ \ \ \beta = \frac{c}{2w}+\frac{1}{w}|\partial w|^2,  \ w=e^u.$$
The first equation says that $\beta \omega_0 =\sqrt{-1}\partial \overline{\partial}w$, where $\omega_0$ is the K\"ahler form of $h_0$. Taking exterior differentiation, we get $d\beta \wedge \omega_0=0$ hence $d\beta =0$ and $\beta$ is a constant.

Take deck transformation $\gamma$ such that $r=r_{\gamma} \neq 1$. By (\ref{eq:deck}) we have $\gamma^{\ast}\omega_0=r^2\omega_0$. On the other hand, since $w^{-2}h_0 = e^{-2u}\widetilde h = \widetilde g$ is invariant under $\gamma$, we get $\gamma^{\ast}w=rw$, or equivalently,
\begin{equation} \label{eq:gammaw}
w(rUz)=rw(z).
 \end{equation}
 Now applying $\gamma^{\ast}$ on both sides of the equation $\beta \omega_0=\sqrt{-1}\partial \overline{\partial}w$, we get $r^2\partial \overline{\partial}w = r\partial \overline{\partial}w$. As $r>0$ and $r\neq 1$, we conclude that $\partial \overline{\partial}w=0$. So $w$ is a pluriharmonic function on the simply-connected manifold ${\mathbb C}^n\setminus \{0\}$, hence is the real part of a holomorphic function, which can be extended across the origin by Hartogs Theorem, hence $w$ can be extended across the origin. By letting $z$ approaching $0$ in  (\ref{eq:gammaw}), we get $w(0)=rw(0)$ hence $w(0)=0$. So the origin becomes a maximum point for the non-positive harmonic function $-w$ on ${\mathbb C}^n$, which is a contradiction to the strong maximum principle. This shows that the strict case cannot occur, and Theorem \ref{thm1} is proved.
\end{proof}


\begin{proof}[{\bf Proof of Theorem \ref{thm2}.}]
The first half of the proof is strictly analogous to that in the proof of Theorem \ref{thm1}. Let $(M^n,g)$ be a compact locally conformally K\"ahler manifold such that its Bismut holomorphic sectional curvature is constant: $H^b=c$. If $g$ is globally conformally K\"ahler, then again by Proposition \ref{prop2} we know that $g$ must be K\"ahler, hence is a complex space form. Now suppose that  the Lee form is not exact on $M^n$. Again denote by $\pi: \widetilde{M} \rightarrow M$ the universal cover, $\widetilde{g}=\pi^{\ast}\!g$ the pull-back, and $\widetilde{h}=e^{2u}\widetilde{g}$ the K\"ahler metric on $\widetilde{M}$. By Proposition \ref{prop1} we know that $\widetilde{h}$ is Bochner-K\"ahler, hence by scaling the metric $g$ by a suitable constant multiple when necessary, Theorem \ref{thm:uniformization} applies to give us a holomorphically isometric identification: $(\widetilde M, \widetilde h)\cong ({\mathbb C}^n\setminus \{ 0\}, h_0)$, where $h_0$ is the standard Euclidean metric and the deck transformations are by (\ref{eq:deck}). Since the Bismut holomorphic sectional curvature of $\widetilde g$ is a constant $c$, we have formula (\ref{eq:Ab}) and (\ref{eq:Ab2}), where $g$ and $h$ are replaced by $\widetilde g$ and $\widetilde h$, respectively. Since $\widetilde h$ can be identified with $h_0$ and is flat, we get $A^b=0$ by the injectivity of $L_{\widetilde{h}}$, that is, we have
$$ \alpha \widetilde h - \chi_f =0, \ \ \ \alpha = c+\frac{1}{f}|\partial f|^2, \ \ f=e^{2u}.$$
Again the first equation says that $\alpha \omega_0 =\sqrt{-1}\partial \overline{\partial}f$, and by taking exterior differentiation we conclude that $d\alpha \wedge \omega_0=0$, hence $\alpha$ is a constant. Since $\omega_0=\sqrt{-1} \partial \overline{\partial} |z|^2$, we get
$$ \partial \overline{\partial} F=0, \ \ \ \ \ F=f-\alpha|z|^2.$$
On the simply-connected manifold ${\mathbb C}^n\setminus \{ 0\}$, the pluriharmonic function $F$ can be written as twice of the real part of a holomorphic function $\Phi$, which can be extended across the origin by Hartogs Theorem, that is,
$F=\Phi + \overline{\Phi}$, where $\Phi$ is holomorphic on ${\mathbb C}^n$.

For any deck transformation $\gamma$, which sends $z$ to $rUz$, we have $\gamma^{\ast}h_0=r^2h_0$ and $\gamma^{\ast}f=r^2f$ as $\widetilde g=f^{-1}h_0$ is invariant under $\gamma$. Hence $\gamma^{\ast}F=r^2F$, namely,
$$ F(rUz)- r^2F(z) = 0. $$
In particular, by taking $\gamma$ with $r\neq 1$ we conclude that $F(0)=0$. So $\Phi (0)$ has zero real part, and by subtracting a pure imaginary constant from $\Phi$, we may assume that $\Phi (0)=0$.

For any fixed $\gamma$ with $r\neq 1$,  the holomorphic function $\Phi(rUz)-r^2\Phi (z)$ on ${\mathbb C}^n$ has vanishing real part, so it must be a constant, and the constant is zero by our choice of $\Phi (0)=0$. So we have
$$ \Phi (rUz) = r^2 \Phi (z). $$
Write $\Phi = \sum_{m=1}^{\infty } \Phi_m$ as a power series where $\Phi_m$ is a homogeneous polynomial of degree $m$ in $z=(z_1, \ldots , z_n)$.  The above equation leads to
$$ \Phi_m (Uz) = r^{2-m}\Phi_m (z). $$
Since $U$ is unitary and preserves the unit sphere $\{ z\in {\mathbb C}^n \mid |z|=1\}$, we have
$$  \max_{|z|=1} |\Phi_m(Uz)| = \max_{|z|=1} |\Phi_m (z)|. $$
Therefore $\Phi_m=0$ for any positive $m\neq 2$, hence $\Phi$ is a homogeneous quadratic polynomial in $z$ and
\begin{equation}  \label{eq:f}
f = \alpha |z|^2 + F = \alpha |z|^2 + \,^t\!zSz + \overline{^t\!zSz},
\end{equation}
where $S$ is a symmetric $n\times n$ matrix and here we have written $z$ as a column vector. Write $q(z)=\,^t\!zSz$ for the homogeneous quadratic polynomial. Since  on ${\mathbb C}^n\setminus \{0\}$ we have $f>0$ and
\begin{equation} \label{eq:alphac}
 \alpha = c+ \frac{|\partial f|^2 }{f},
\end{equation}
so $\alpha$ is real. We want to derive from these restrictions that $q=0$ and $c=0$. By (\ref{eq:f}) and (\ref{eq:alphac}), we get
\begin{eqnarray*}
 \frac{\partial f}{\partial z_i} &= & \alpha \overline{z}_i + 2\sum_j S_{ij}z_j, \\
 |\partial f|^2 & =&  \alpha^2|z|^2  + 4\,^t\!zS\overline{S}\overline{z} + 2\alpha q(z) + 2 \alpha \overline{q(z)} \\
&=&  (\alpha -c)\big( \alpha |z|^2 + q(z) + \overline{q(z)} \big).
 \end{eqnarray*}
So by looking at the $(2,0)$ and $(1,1)$ parts in the last equation above, we obtain
\begin{equation} \label{eq:3parts}
(\alpha +c)q(z)=0, \ \  \ \ -\alpha c I = 4 S \overline{S}.
\end{equation}
Now assume that $q$ is not identically zero. Then we have $\alpha +c=0$. Plug it into the second equation of (\ref{eq:3parts}) we get $\alpha^2I=4S\overline{S}$. Since $S$ is symmetric, there exists unitary matrix $V$ such that $\,^t\!VSV=S'$ is diagonal with diagonal values being real and non-negative. Plug into the previous equation we know that $S'=\frac{|\alpha|}{2}I$. Hence $S=\frac{|\alpha|}{2} \overline{V} V^{-1}$. Let us consider the point $z_0=\sqrt{-1}Ve_1$ where $e_1=\,^t\!(1, 0, \cdots , 0)$. We have $|z_0|^2=1$ and
$$ q(z_0) = - \,^t\!e_1\mbox{}^t\!V S V e_1 = -\frac{|\alpha|}{2}, \ \ \ \ \ f(z_0) = \alpha  -\frac{|\alpha|}{2} -\frac{|\alpha|}{2} \le 0. $$
This contradicts to our assumption that $f>0$ on $\widetilde M$, so we must have $q=0$. Now by $f=\alpha |z|^2$ we know  that $\alpha >0$ hence by the second equation of (\ref{eq:3parts}) we get $c=0$. By the definition of $f$, $\widetilde g = f^{-1}h_0 = \frac{1}{\alpha |z|^2} h_0$ is a constant multiple of the standard Hopf metric whose K\"ahler form is
$$ \alpha \pi^{\ast} \omega_g = \frac{\sqrt{-1}}{|z|^2} \partial \overline{\partial} |z|^2. $$

It remains to show that $(M^n,g)$ is an isosceles Hopf manifold. Let $\Gamma=\operatorname{Deck}(\pi)$. By Theorem~\ref{thm:uniformization}, every $\gamma\in\Gamma$ is of the form $\gamma(z)=r_\gamma U_\gamma z$, where $r_\gamma>0$ and $U_\gamma\in U(n)$. Define $\ell:\Gamma\to\mathbb R$ by $\ell(\gamma)=\log r_\gamma$. Since $r_{\gamma_1\gamma_2}=r_{\gamma_1}r_{\gamma_2}$, the map $\ell$ is a group homomorphism.

We first claim that $\ker\ell$ is finite. Indeed, if $\gamma\in\ker\ell$, then $r_\gamma=1$, so $\gamma$ is represented by a unitary transformation $U_\gamma\in U(n)$. Since $\Gamma$ acts properly discontinuously on $\mathbb C^n\setminus\{0\}$, $\ker\ell$ is a discrete subgroup of the compact group $U(n)$, and hence is finite.

We next claim that $\ell(\Gamma)$ is a discrete subgroup of $\mathbb R$. Suppose otherwise. Then there exists a sequence of distinct elements $\gamma_j\in\Gamma$ such that $\ell(\gamma_j)\to0$. Write $\gamma_j(z)=e^{\ell(\gamma_j)}U_jz$, where $U_j\in U(n)$. By compactness of $U(n)$, after passing to a subsequence, we may assume that $U_j$ converges in $U(n)$. It follows that the nontrivial deck transformations $\delta_j:=\gamma_{j+1}\gamma_j^{-1}$ converge to the identity transformation, contradicting the proper discontinuity of the deck action. Hence $\ell(\Gamma)$ is discrete.

Since the Lee form is not exact, Theorem~\ref{thm:uniformization} guarantees that $r_\gamma\neq1$ for some $\gamma\in\Gamma$. Thus $\ell(\Gamma)$ is a nonzero discrete subgroup of $\mathbb R$, and hence $\ell(\Gamma)=a\mathbb Z$ for some $a>0$. Choose $\gamma_0\in\Gamma$ such that $\ell(\gamma_0)=a$. For any $\gamma\in\Gamma$, there exists $m\in\mathbb Z$ such that $\ell(\gamma)=ma$, and therefore $\gamma\gamma_0^{-m}\in\ker\ell$. Hence $\Gamma=(\ker\ell)\langle\gamma_0\rangle$. Since $\ker\ell$ is finite, $\langle\gamma_0\rangle$ has finite index in $\Gamma$. Consequently,
$$ \frac{\mathbb C^n\setminus\{0\}}{\langle\gamma_0\rangle} \longrightarrow \frac{\mathbb C^n\setminus\{0\}}{\Gamma}\simeq M $$
is a finite unbranched covering.

Replacing $\gamma_0$ by its inverse if necessary, we may write $\gamma_0(z)=rUz$, where $0<r<1$ and $U\in U(n)$. Since $U$ is unitary, there exists $V\in U(n)$ such that $$
V^{-1}UV
=
\operatorname{diag}
\bigl(e^{\sqrt{-1}\theta_1},\ldots,e^{\sqrt{-1}\theta_n}\bigr).
$$
After the unitary change of coordinates $w=V^{-1}z$, the generator $\gamma_0$ becomes
$$
w\longmapsto
\operatorname{diag}
\bigl(
re^{\sqrt{-1}\theta_1},\ldots,re^{\sqrt{-1}\theta_n}
\bigr)w.
$$
Thus all eigenvalues of the generating contraction have the same modulus $r\in(0,1)$.

Moreover,
$$
\gamma_0^*
\left(\frac{1}{\alpha|z|^2}h_0\right)
=
\frac{1}{\alpha|rUz|^2}\,r^2h_0
=
\frac{1}{\alpha|z|^2}h_0.
$$
Hence the metric descends to $(\mathbb C^n\setminus \{0\})/\langle\gamma_0\rangle$ as a constant multiple of the standard Hopf metric. Therefore, $(M^n,g)$ is an isosceles Hopf manifold. This completes the proof of Theorem \ref{thm2}.
\end{proof}

\vspace{0.3cm}

\noindent\textbf{Declaration on competing interests.}
All authors declare that there are no competing interests for this article.

\end{document}